\documentclass[12pt]{amsart}
\usepackage{todonotes}
\usepackage[margin=1.05in]{geometry}
\usepackage{amsmath,amssymb,amsthm,mathtools}
\usepackage{mathdots}
\usepackage[T1]{fontenc}
\usepackage{lmodern}
\usepackage{microtype}
\usepackage{thmtools}
\usepackage[colorlinks=true,linkcolor=blue,citecolor=blue,urlcolor=blue]{hyperref}
\usepackage[nameinlink]{cleveref}
\usepackage{xcolor}

\newtheorem{theorem}{Theorem}[section]
\newtheorem{proposition}[theorem]{Proposition}
\newtheorem{lemma}[theorem]{Lemma}
\newtheorem{corollary}[theorem]{Corollary}
\newtheorem{algorithm}[theorem]{Algorithm}

\newtheorem{hypothesis}{Hypothesis}

\theoremstyle{definition}
\newtheorem*{definition}{Definition}
\newtheorem*{example}{Example}
\newtheorem*{question}{Question}
\newtheorem*{answer}{Answer}
\theoremstyle{remark}
\newtheorem*{remark}{Remark}

\newcommand{\F}{\mathbb F}
\newcommand{\Z}{\mathbb Z}

\DeclareMathOperator{\wt}{wt}

\newcommand{\ts}{\textsuperscript}

\newcommand{\Col}{\mathrm{Col}}
\newcommand{\Diag}{\mathrm{Diag}}
\newcommand{\Tdk}{\mathcal{T}_{d,k-1}}

\DeclareMathOperator{\Slots}{Slots}

\title{Thakur's Hypotheses on Power Sums over $\F_q[t]$}
\author{Evan Chen, and Ken Ono}
\address{Axiom Math, 124 University Avenue, Palo Alto, CA 94301}
\email{evan@axiommath.ai}
\email{ken@axiommath.ai}
\date{June 15, 2026}

\usepackage{tikz}
\usetikzlibrary{shapes.geometric, arrows, arrows.meta}

\tikzstyle{proved} = [rectangle, rounded corners,
text width=4.5cm,
minimum height=1cm,
text centered,
draw=black,
fill=green!30]

\tikzstyle{assumed} = [rectangle, rounded corners,
text width=4.5cm,
minimum height=1cm, text centered,
draw=black,
fill=cyan!20]

\tikzstyle{result} = [rectangle, rounded corners,
text width=3.5cm,
minimum height=0.8cm,
text centered,
draw=black,
fill=cyan!20]

\tikzstyle{arrow} = [thick,->,>=stealth, -{Latex[length=3mm]}]


\begin{document}

\begin{abstract}
In his 2009 paper \cite{Thakur2009}, Thakur posed three
conjectural hypotheses for the degrees of the power sums
\[
S_d(k)=\sum_{\substack{a\in \mathbb F_q[t] \text{ monic}\\ \deg a=d}} a^{-k}
\qquad \text{and}\qquad
s_d(k)=-\deg_t S_d(k).
\]
For prime fields $q=p$, we prove Hypotheses H1 and H2,
giving a unique greedy description of the extremal term in Carlitz's formula and establishing the recursion
\[
s_d(k)=s_{d-1}(s_1(k))+s_1(k).
\]
As consequences, the prime-field recursion gives the strict Newton-polygon
convexity used in the prime-field Carlitz--Goss Riemann-hypothesis theorem,
and it recovers Thakur's nonvanishing theorem for positive multizeta values over $\mathbb F_p[t]$.
We also prove Hypothesis H3 for all finite fields $q=p^f$, establishing the monotonicity
\[
s_d(k)<s_d(k+1)\qquad (p\nmid k).
\]
We provide Lean formalizations of the arguments in this paper,
generated by AxiomProver.
\end{abstract}

\maketitle

\section{Introduction and Statement of Results}
In this paper, $q \coloneq p^f$ is a power of a prime,
and we work over the finite field $\F_q$.
In \Cref{sec:H1,sec:H2} we specialize to $f = 1$;
the hypothesis $q=p$ means that base-$p$ and base-$q$ expansions coincide.

\subsection{The Carlitz--Goss zeta function and its degree-$d$ parts}
We define the Carlitz--Goss zeta function in terms of its degree-$d$ pieces as follows.
Let $A \coloneq \F_q[t]$ and for $d\geq 0$, write $A_d^+$ for the monic polynomials in $A$ of degree $d$.
For an integer $k$, put
\begin{equation}\label{eq:intro-Sd}
  S_d(k) \coloneq \sum_{a\in A_d^+}\frac{1}{a^k}\in \F_q(t)
  \qquad {\text {\rm and}}\qquad
  s_d(k) \coloneq -\deg_t S_d(k) \ge 0.
\end{equation}
We use the convention $\deg_t(0)=-\infty$, so that $s_d(k)=+\infty$ if $S_d(k)=0$.
Also $A_0^+=\{1\}$, and so we have $S_0(k)=1$ and $s_0(k)=0$.
Then the sums \eqref{eq:intro-Sd} are the degree-$d$ pieces of the Carlitz--Goss zeta values,
\begin{equation}
  \zeta(k) \coloneq \sum_{d\geq 0}S_d(k),
\end{equation}
and of the function-field multizeta values
\[
  \zeta(s_1,\ldots,s_r) \coloneq \sum_{d_1>\cdots>d_r\geq 0}
  S_{d_1}(s_1)\cdots S_{d_r}(s_r).
\]
For general background on function-field arithmetic and Carlitz--Goss zeta functions,
see \cite{Goss1996} and  \cite{Thakur2004}.
Although each summand $a^{-k}$ has transparent degree,
the degree of $S_d(k)$ is subtle because of cancellation in characteristic $p$.
Thakur \cite{Thakur2009} studies this cancellation problem through the integers $s_d(k)$,
the orders of the power sums at the infinite place.

One motivation for the degree problem is its relation to the zero distribution
of our zeta functions.
The Carlitz--Goss zeta function is naturally a two-variable function;
after fixing the $p$-adic weight parameter,
its zeros in the remaining variable a priori lie in the completed algebraic
closure $C_\infty$ of $\F_q((1/t))$.
Following the standard terminology in this subject,
the Riemann-hypothesis analogue for $\F_q[t]$ asserts that these zeros are
simple and lie on the ``real line'' $K_\infty=\F_q((1/t))$.  This theorem is known.
Wan \cite{Wan1996} proved it over $\F_p[t]$,
and Diaz--Vargas \cite{DiazVargas1996} gave another proof in the prime-field case,
and Sheats \cite{Sheats1998} proved the general finite-field case.

\subsection{On Hypothesis H1, H2, and H3}
The point of the present paper is to study three hypotheses, H1, H2 and H3,
posed by Thakur in \cite[Appendix A.12,
with H1 and H2 formulated in Sections 2.3.1 and 3.2]{Thakur2009} regarding these power sums.
We prove the prime-field forms of H1 and H2, and we prove H3 over arbitrary finite fields.
We then also explain why several important
consequences follow directly from the leading-term structure of the power sums.

For example, H2 gives a short recursion for the numbers $s_d(k)$.
This recursion immediately yields the strict Newton-polygon convexity used by
Thakur to deduce the prime-field Carlitz--Goss Riemann-hypothesis analogue,
and it also gives nonvanishing of all positive function-field multizeta values.
Therefore, the first nonzero term of $S_d(k)$ carries both computational and arithmetic information.

The proofs are organized around a common combinatorial mechanism.
In all three proofs, the power sums $S_d(k)$ are parametrized
via carry-free digit decompositions.
The central problem is therefore to
identify which decomposition gives the extremal degree, and
show that this decomposition is unique.
When $q = p$, this uniqueness is governed by a so-called
\emph{greedy} assignment of the individual base-$p$ place-value terms,
meaning that the relevant place-value terms are ordered
and then placed successively in the smallest permitted slot allowed by constraints.
A second, equivalent description in terms of reciprocal digit slots gives the recursion predicted by H2.
For H3, the argument is different and relies on Sheats' uniqueness theorem.

\subsection{Hypothesis H1 and its consequences}
Carlitz's generating function \cite{Carlitz1935}, in the form used in \cite[Section 2.3]{Thakur2009},
gives an expansion for $S_d(k+1)$ defined as follows.
Writing
\[ \ell_n \coloneq \prod_{i=1}^n(t-t^{q^i}), \]
we have
\begin{equation}
  S_d(k+1) = \ell_d^{-1} \sum_{\substack{k=k_0+k_1q+\cdots+k_dq^d \\ k_i \ge 0}}
  \binom{k_0 + k_1 + \dots + k_d}{k_0, \dots, k_d}
  \prod_{i=0}^d \left( d_i \ell_{d-i}^{q^i} \right)^{-k_i}.
  \label{eq:carlitz-full}
\end{equation}
For the purposes of Hypothesis H1, we only need to consider the maximum-degree
terms of \eqref{eq:carlitz-full}.
To that end, for $0\leq i\leq d$, define
\begin{equation}\label{eq:intro-bid}
  b(i,d) \coloneq -\deg(d_i \ell_{d-i}^{q^i})
  = -\left(\frac{q^{d+1}-q^{i+1}}{q-1}+iq^i\right) \le 0.
\end{equation}
Then the point is that \eqref{eq:carlitz-full} is a finite sum indexed by decompositions
\begin{equation}\label{eq:intro-carlitz-decomp}
  k=k_0+k_1q+\cdots+k_dq^d,
  \qquad k_i\in\Z_{\geq 0}
\end{equation}
where the degree of the summand indexed by $(k_0,\ldots,k_d)$ is
\begin{equation}\label{eq:intro-degree-contribution}
  b(0,d)+\sum_{i=0}^d k_i b(i,d).
\end{equation}
The term $b(0,d)$ is independent of the decomposition.
Meanwhile, we also know which coefficients of \eqref{eq:carlitz-full} are nonzero;
by Lucas' theorem (\Cref{lem:Lucas-binomial} below), the multinomial coefficient
$\binom{k_0 + k_1 + \dots + k_d}{k_0, \dots, k_d}$
is nonzero modulo $p$ exactly when the addition
\[ k_0+k_1+\cdots+k_d \]
has no carry in base-$p$.
The possible cancellation at the leading degree is
thus controlled by the following uniqueness question.
\begin{hypothesis}
\label{H1}
In Carlitz's expansion for $S_d(k+1)$,
among the nonzero summands indexed by decompositions \eqref{eq:intro-carlitz-decomp},
there is a unique summand of maximum $t$-degree.
\end{hypothesis}

Our first theorem proves H1 in the prime-field case.
Its role is to guarantee that the top-degree term in
Carlitz's formula \eqref{eq:carlitz-full} is not merely detectable, but unique.

\begin{theorem}[H1 for prime fields]\label{thm:intro-H1}
Let $q=p$ be prime, $d \ge 0$, and $k \ge 0$.
Then \Cref{H1} holds: among all decompositions $k = \sum k_i q^i$
for which the addition $\sum k_i$ has no carries in base-$p$,
there is a unique choice of decomposition maximizing
\[ b(0,d)+\sum_{i=0}^d k_i b(i,d). \]
\end{theorem}

The uniqueness statement immediately turns the formal expansion into an effective degree formula.
We record this consequence separately because it is the computational form of H1 used later.

\begin{corollary}[Degree computation from Carlitz's formula]\label{cor:intro-H1}
Assume $q=p$.  For each $d\geq 0$ and $k\geq 0$,
let $(k_0^*,\ldots,k_d^*)$ be the unique maximizing decomposition in \Cref{thm:intro-H1}.  Then
\[
  s_d(k+1)=-\left(b(0,d)+\sum_{i=0}^d k_i^* b(i,d)\right).
\]
In particular, the leading term in Carlitz's formula for $S_d(k+1)$ cannot cancel.
\end{corollary}

This is why H1 was raised in \cite[Section 2.3.1]{Thakur2009}.
It converts the problem of finding the degree of a characteristic-$p$ rational function into a finite,
explicit, carry-free optimization problem.

\begin{remark}[The non-prime-field case]\label{rem:intro-nonprime}
The proof of \Cref{thm:intro-H1} uses the equality of two bases: Lucas' theorem sees base-$p$,
while Carlitz's decomposition is in powers of $q$.  When $q=p$, these are the same.
When $q=p^f$ with $f>1$, they are not.
There are several extension-field pathologies, especially for $q=4$, in \cite[Appendix A.11]{Thakur2009}.
These examples do not by themselves disprove H1,
but they show that the proof in this paper is genuinely prime-field in nature.
We make no claim here about H1 for non-prime finite fields.
\end{remark}

\subsection{Hypothesis H2 and its consequences}
Assume $k>0$ and $d \ge 1$.
There is another way besides \eqref{eq:carlitz-full} to rewrite $S_d(k)$,
which is more recursive in nature.
From \cite[Section 3.2]{Thakur2009}, we have the calculation
\begin{equation}
  \label{eq:prelim-H2-expansion}
  S_d(k)
  =-\sum_{\substack{i\geq 1\\(q-1)\mid i}}
  \binom{-k}{i}\frac{S_{d-1}(k+i)}{t^{k+i}}.
\end{equation}
We shift the summation variable of
\eqref{eq:prelim-H2-expansion} so that its first term is nonzero.
In fact, \cite[Section 3.1]{Thakur2009}
shows that the first nonzero term occurs at $i = s_1(k) - k$.
Hence, we let $j = i - (s_1(k) - k)$
and change summation to $j \ge 0$ with $(q-1) \mid j$.
Further rewriting the binomial coefficient as
\[ \binom{-k}{i} = (-1)^i \binom{k+i-1}{k-1}
  = (-1)^{s_1(k)-k+j} \binom{s_1(k) + j - 1}{k-1} \]
we obtain
\begin{equation}
  \label{eq:intro-H2-expansion}
  S_d(k)
  = (-1)^{s_1(k)-k+1} \sum_{\substack{j\geq 0\\(q-1)\mid j}}
  (-1)^j \binom{s_1(k)+j-1}{k-1} \frac{S_{d-1}(s_1(k)+j)}{t^{s_1(k)+j}}
\end{equation}
together with a promise that the first summand (at $j=0$) has nonzero binomial coefficient.
In \eqref{eq:intro-H2-expansion}, the degree of a nonzero summand is
$-(s_{d-1}(s_1(k)+j)+s_1(k)+j)$ for each $j$.

In the prime-field case, \Cref{H2} further asserts that
the $j=0$ term also controls the leading term of \eqref{eq:intro-H2-expansion}, i.e.:
\begin{hypothesis}
\label{H2}
Assume $q=p$ is prime and $k>0$.  Among all $j \geq 0$ satisfying
\[
  (p-1)\mid j\qquad \text{and} \qquad \binom{s_1(k)+j-1}{k-1}\not\equiv 0\pmod p,
\]
the quantity
\[
  s_{d-1}(s_1(k)+j)+s_1(k)+j
\]
is minimized uniquely at $j=0$.
\end{hypothesis}

\begin{remark}
  The restriction to prime fields is essential in \Cref{H2}.
  Thakur gives positive-characteristic extension-field counterexamples to this minimization principle;
  for instance, when $q=4$ and $k=75$,
  the minimum in the H2 problem occurs at $j=12$ rather than at $j=0$ \cite[Appendix A.11]{Thakur2009}.
\end{remark}

The next theorem verifies this minimization principle over prime fields.

\begin{theorem}[H2 for prime fields]\label{thm:intro-H2}
  Let $q=p$ be prime, $k>0$, and $d\geq 1$.  Then \Cref{H2} holds.
\end{theorem}

The promised recursion is the most useful way to read H2.
It reduces the degree of a degree-$d$ power sum to the corresponding
degree-$(d-1)$ problem at the shifted exponent $s_1(k)$.

\begin{corollary}[Recursion]\label{cor:intro-H2}
Let $q=p$ be prime.  For all $k>0$ and $d\geq 1$, we have
\begin{equation}\label{eq:intro-recursion}
  s_d(k)=s_{d-1}(s_1(k))+s_1(k).
\end{equation}
\end{corollary}

The importance of \Cref{cor:intro-H2} is that it packages the
cancellation problem into a one-step recursion.
Iterating \eqref{eq:intro-recursion} expresses $s_d(k)$ in terms of successive values of $s_1$,
which is the only initial input.
The proof passes through the reciprocal slot formula \Cref{thm:H2-slot-formula},
which gives a direct degree formula in the prime-field case and may be useful independently.
Two standard consequences are recalled next so that the arithmetic content of the recursion is explicit.

\begin{corollary}[Strict convexity of the degree sequence]\label{cor:intro-convexity}
Let $q=p$ be prime and $k>0$.  Then
\[
  s_d(k)<s_{d+1}(k)\qquad(d\geq 0),
\]
and the degree jumps are strictly increasing:
\[
  s_d(k)-s_{d-1}(k)<s_{d+1}(k)-s_d(k)\qquad(d\geq 1).
\]
\end{corollary}

The Carlitz--Goss Riemann-hypothesis analogue is already known:
Wan and Diaz--Vargas proved the prime-field case, and Sheats proved the theorem for arbitrary finite fields.
The next corollary records the point relevant here:
H2 supplies the strict Newton-polygon convexity input in Thakur's proof,
and Thakur's criterion then gives the prime-field RH conclusion.

\begin{corollary}[Prime-field Carlitz--Goss RH via Thakur's criterion]\label{cor:intro-RH}
Let $q=p$ be prime.  The Riemann-hypothesis for the Carlitz--Goss zeta function over $\F_p[t]$ is true.
Moreover, the fixed-weight Carlitz--Goss zeta function has only simple zeros lying in $K_\infty=\F_p((1/t))$.
\end{corollary}

The same monotonicity also separates the leading term in each positive multizeta sum.
This gives the following nonvanishing statement,
recovering in the prime-field case the known theorem of Thakur \cite[Theorem~4]{Thakur2009}.

\begin{corollary}[Nonvanishing of positive multizeta values]\label{cor:intro-MZV}
Let $q=p$ be prime and let $s_1,\ldots,s_r$ be positive integers.  Then the function-field multizeta value
\[
  \zeta(s_1,\ldots,s_r)=\sum_{d_1>\cdots>d_r\geq 0}
  S_{d_1}(s_1)\cdots S_{d_r}(s_r)
\]
is nonzero.
\end{corollary}

This explains the role of H2 in \cite{Thakur2009}.
The hypothesis is not merely a local assertion about \eqref{eq:intro-H2-expansion}.
It produces the recursion from which the zero-distribution and multizeta consequences follow.

\subsection{Hypothesis H3 and its consequences}
The third hypothesis concerns monotonicity in the exponent.
Unlike H1 and H2, this statement is not restricted to prime fields.

\begin{hypothesis}
\label{H3}
For $d\geq 1$ and $k>0$ with $p\nmid k$, one has
\[
  s_d(k)<s_d(k+1).
\]
\end{hypothesis}
In other words, the valuation at infinity of the degree-$d$
power sum strictly increases when the exponent is raised from $k$ to $k+1$; that is,
\[
  -\deg_t S_d(k)<-\deg_t S_d(k+1).
\]
The restriction $d\geq 1$ is necessary because $s_0(k)=0$ for every $k$.

We prove H3 in full generality:
\begin{theorem}[H3]\label{thm:intro-H3}
Let $q=p^f$ be any prime power.  For every $d\geq 1$ and every $k>0$ with $p\nmid k$,
\[
  s_d(k)<s_d(k+1).
\]
\end{theorem}

The rest of this paper is organized as follows (see \Cref{fig:skeleton}).
\Cref{sec:prelim} fixes notation and recalls the
elementary preliminaries needed for subsequent proofs.
\Cref{sec:H1} proves \Cref{H1} for prime fields by reducing Carlitz's formula to
a finite assignment problem and identifying its unique greedy optimum.
\Cref{sec:H2H3} provides a formula for $S_d(k)$ used in both H2 and H3,
and applies Sheats' theorem.
\Cref{sec:H2} proves \Cref{H2} for prime fields by introducing reciprocal digit
slots and proving the block-minimization lemma that yields the recursion in \cite{Thakur2009} for $s_d(k)$.
\Cref{sec:H3} proves \Cref{H3} for arbitrary finite fields.

\begin{figure}[hb]
  \centering
  \begin{tikzpicture}[x=1cm, y=0.8cm]
    \node (lucas) [result] at (0,6) {Lucas (\S\ref{sec:lucas})};
    \node (pow) [result] at (5,6) {Pow sum (\S\ref{sec:power_sum})};
    \node (block) [result] at (10,6) {Block ineq (\S\ref{sec:block})};
    \node (lem41) [result, text width = 5cm] at (2.5,4) {Expansion (\Cref{lem:multinomial_expansion})};
    \node (minform) [result, text width = 5cm] at (2.5,2) {Min formula (\Cref{lem:H3-minformula})};
    \node (H1) at (0,0) [result] {\Cref{H1} (\S\ref{sec:H1})};
    \node (H2) at (5,0) [result] {\Cref{H2} (\S\ref{sec:H2})};
    \node (H3) at (10,0) [result] {\Cref{H3} (\S\ref{sec:H3})};
    \draw [arrow] (lucas) -- (lem41);
    \draw [arrow] (pow) -- (lem41);
    \draw [arrow] (lem41) -- node[anchor=east] {using Sheats' theorem} (minform);
    \draw [arrow] (minform) -- (H2);
    \draw [arrow] (minform) -- (H3);
    \draw [arrow] (block) -- (H2);
  \end{tikzpicture}
  \caption{An outline of the components of this paper.}
  \label{fig:skeleton}
\end{figure}
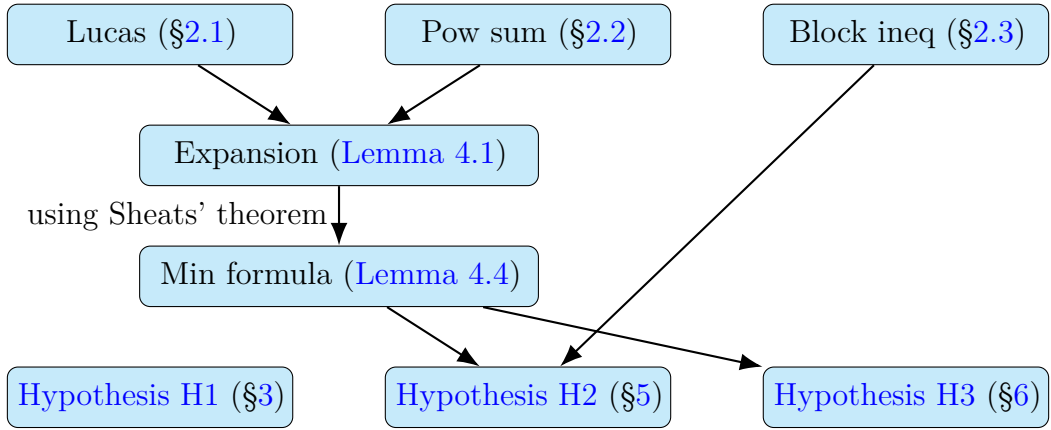

\section{Preliminaries}\label{sec:prelim}
We state the elementary preliminaries required for this paper.

\subsection{Lucas theorem and its variants}
\label{sec:lucas}
We record the elementary digit facts used throughout.

\begin{definition}[Carry-free addition]\label{def:carryfree}
For nonnegative integers $x_1,\ldots,x_r$, we write
\[
  x_1\oplus x_2\oplus\cdots\oplus x_r
\]
when the addition is carry-free in base-$p$.  Equivalently, if we have
\[
  x_i=\sum_{e\geq 0}x_{i,e}p^e
  \qquad {\text {\rm and}}\qquad 0\leq x_{i,e}\leq p-1,
\]
then we have
\[
  \sum_{i=1}^r x_{i,e}\leq p-1
  \qquad\text{for every }e.
\]
\end{definition}

The first tool is Lucas' theorem in its binomial form,
which provides a bridge between nonvanishing modulo $p$ and carry-free decompositions.

\begin{lemma}[Lucas' theorem]\label{lem:Lucas-binomial}
Let $m,n\geq 0$, and write
\[
  m=\sum m_ep^e\qquad {\text {\rm and}}\qquad n=\sum n_ep^e,
  \qquad 0\leq m_e,n_e\leq p-1.
\]
Then
\[
  \binom{n}{m}\equiv\prod_e\binom{n_e}{m_e}\pmod p.
\]
In particular, $\binom{n}{m}\not\equiv 0\pmod p$ if and only if $m_e\leq n_e$ for every $e$;
that is, if $m \oplus (n-m)$ is carry-free in base $p$.
\end{lemma}

We will mostly use the following multinomial form,
since Carlitz's formula naturally involves multinomial coefficients.

\begin{lemma}[Lucas' theorem, multinomial form]\label{lem:mult-Lucas}
Let $R=r_1+\cdots+r_m$.  Then
\[
  \binom{R}{r_1,\ldots,r_m}\not\equiv 0\pmod p
\]
if and only if $r_1 \oplus r_2 \oplus \dots \oplus r_m$ is carry-free in base-$p$.
\end{lemma}

\begin{proof}
Write the multinomial coefficient as a product of binomial coefficients:
\[
  \binom{R}{r_1,\ldots,r_m}
  =\binom{r_1+\cdots+r_m}{r_m}
   \binom{r_1+\cdots+r_{m-1}}{r_{m-1}}\cdots
   \binom{r_1+r_2}{r_2}.
\]
Each factor is nonzero modulo $p$ exactly when the indicated addition has no carry.
This is equivalent to saying that, at each base-$p$ digit position,
the digits of $r_1,\ldots,r_m$ add to at most $p-1$.
\end{proof}

\subsection{Finite field power sums}
\label{sec:power_sum}
The other elementary input is the standard finite-field power sum.
It is responsible for the divisibility conditions by $q-1$ that appear
throughout the later minimization problems.

\begin{lemma}[Finite-field power sums]\label{lem:field-sums}
For $m\geq 0$,
\[
  \sum_{x\in\F_q}x^m=
  \begin{cases}
    -1, & m>0\text{ and }(q-1)\mid m,\\
    0, & \text{otherwise}.
  \end{cases}
\]
Here we consider $0^0=1$ for $m=0$.
\end{lemma}

\begin{proof}
If $m=0$, the sum is $q$, which is $0$ in characteristic $p$.
If $m>0$, the zero term contributes $0$,
and the nonzero elements of $\F_q$ form a cyclic group of order $q-1$.
The sum of the $m$th powers in this cyclic group is $0$ unless $(q-1)\mid m$;
in that case it is instead $q-1=-1$ in $\F_q$.
\end{proof}

\subsection{Block-rearrangement inequality}
\label{sec:block}
This part will only be used in \Cref{sec:H2} to prove \Cref{H2}.
Throughout this section fix a positive integer $n$.
Let $M$ be any infinite multiset of positive real numbers, listed in nondecreasing order as
\[ c_1\leq c_2\leq c_3\leq\cdots.  \]
We will refer to the $c_i$ as \emph{slots} (that we imagine putting weights in).

The block-rearrangement inequality is an easy variant of the rearrangement inequality.
Let's describe it in words before introducing the relevant notation.
\begin{question}
  Let $d \ge 1$.
  Suppose we need to pick out $dn$ elements of $M$ and compute a weighted sum,
  where $n$ of the elements have weight $1$,
  $n$ of the elements have weight $2$, \dots,
  and $n$ elements have weight $d$.
  How can we minimize this weighted sum?
\end{question}
\begin{answer}
  This is easy:
  the rearrangement inequality says we should take the $dn$ smallest elements of $M$,
  and pair the largest weights with the smallest numbers.
\end{answer}
We introduce notation corresponding to what we just described.
For $r\geq 1$, let $B_r(M)$ be the $r$th consecutive block of $n$ elements:
\[
  B_r(M) \coloneq \{c_{(r-1)n+1},\ldots,c_{rn}\},
\]
and put
\[
  \beta_r(M) \coloneq \sum_{c\in B_r(M)}c.
\]
For $d\geq 1$, define
\[
  \Phi_d(M) \coloneq d\beta_1(M)+(d-1)\beta_2(M)+\cdots+\beta_d(M),
\]
and set $\Phi_0(M) \coloneq 0$.
The point is that $\Phi_d(M)$ is the minimum we are describing:

\begin{lemma}[Block minimization]\label{lem:H2-block-min}
Let $T_1,\ldots,T_d$ be pairwise disjoint finite submultisets of $M$ with $|T_i|=n$ for every $i$.
Then
\[
  \sum_{i=1}^d i\sum_{c\in T_i}c\geq \Phi_d(M).
\]
Equality holds precisely when the largest weight $d$ receives the first block $B_1(M)$,
the next largest weight receives $B_2(M)$, and so on, up to interchanging equal-valued slots.
\end{lemma}

\begin{proof}
The union $T_1\cup\cdots\cup T_d$ contains $dn$ elements.
A minimizer must use the $dn$ smallest elements of $M$:
if a selected slot $a$ is larger than an unselected slot $b$,
then replacing $a$ by $b$ in the same $T_i$ lowers the weighted sum by $i(a-b)>0$.

Now distribute these first $dn$ slots among the $T_i$.
If $i<j$, $a\in T_i$, $b\in T_j$, and $a<b$,
then swapping $a$ and $b$ changes the contribution from $ia+jb$ to $ib+ja$.  The decrease is
\[
  (ia+jb)-(ib+ja)=(j-i)(b-a)>0.
\]
Thus a minimizer must assign smaller slots to larger weights.
Therefore weight $d$ receives the first block, weight $d-1$ receives the second block, and so on.
This gives the value $\Phi_d(M)$.
The only ambiguity comes from slots with equal numerical values, whose interchange does not affect the sum.
\end{proof}

To prove H2, we need a slightly more flexible form of the block-minimization lemma.
It compares the optimal block choice with the situation in which an initial
admissible submultiset has already been extracted.

\begin{lemma}[Extraction inequality]\label{lem:H2-extraction}
Let $U$ be a finite submultiset of $M$ satisfying
\[
  |U|\geq n,
  \qquad n\mid |U|.
\]
Then
\begin{equation}\label{eq:H2-extraction}
  d\sum_{u\in U}u+\Phi_{d-1}(M\setminus U)\geq \Phi_d(M).
\end{equation}
Equality holds if and only if $|U|=n$ and $U=B_1(M)$, up to interchanging equal-valued slots.
\end{lemma}

\begin{proof}
Give weight $d$ to every slot in $U$.
Among the remaining slots, give weight $d-1$ to the first block of size $n$, weight $d-2$ to the next block,
and so on down to weight $1$; all later slots have weight $0$.
Then the left side of \eqref{eq:H2-extraction} is $\sum_c w(c)c$.

For $1\leq h\leq d$, let $E_h$ be the set of slots with weight at least $h$.
The layer-cake identity for nonnegative integer weights gives
\[
  \sum_c w(c)c=\sum_{h=1}^d\sum_{c\in E_h}c.
\]
Because $|U|\geq n$ and $n\mid |U|$, the set $E_h$ contains at least $(d+1-h)n$ slots:
it contains $U$ and the first $(d-h)n$ slots of $M\setminus U$.  Therefore
\[
  \sum_{c\in E_h}c\geq c_1+c_2+\cdots+c_{(d+1-h)n}.
\]
Summing this inequality over $h=1,\ldots,d$ gives
\[
  \sum_c w(c)c\geq d\beta_1(M)+(d-1)\beta_2(M)+\cdots+\beta_d(M)=\Phi_d(M).
\]

For equality to hold, first $|U|$ must be exactly $n$;
otherwise $E_d=U$ contains at least $2n$ positive slots,
and the comparison with the first $n$ slots is strict.
With $|U|=n$, equality in the $h=d$ inequality forces $U$ to be the first block $B_1(M)$,
up to equal-valued slots.  Conversely, if $U=B_1(M)$, the two sides are exactly equal.
\end{proof}

\section{Proof of \Cref{H1} over prime fields}\label{sec:H1}
Throughout this section, we assume that $q=p$ is prime.
This assumption is used in exactly one conceptual place:
the no-carry condition supplied by Lucas' theorem is now a base-$q$ condition, because $q=p$.

\subsection{The optimization problem}
We unwind the definitions involved in the problem more explicitly now.

\subsubsection{Notation for base-$q$ digits}
For each $k_i$, we let $a_{i,n}$ denote its base-$q$ digits
\[ k_i = \sum_{n\geq 0}a_{i,n}q^n, \qquad 0\leq a_{i,n}\leq q-1.  \]
For $k$ itself, we use $C_n$ for its base-$q$ digits,
\[ k = \sum_{n \ge 0} C_n q^n, \qquad 0 \le C_n \le q-1. \]
For convenience, $a_{i,n} = 0$ if $n < 0$.

\subsubsection{Admissible decompositions}
We will say a decomposition
\[ k=k_0+k_1q+\cdots+k_dq^d \]
is \emph{admissible} if $k_0+\cdots+k_d$ is carry-free in base-$q$.
By \Cref{lem:mult-Lucas},
the admissible decompositions are precisely the decompositions that contribute
nonzero terms to Carlitz's formula.

We can visualize the hypotheses by drawing the grid shown in \Cref{fig:addition_table}.
This is a grid which is infinite to the left,
with columns \dots, $\Col_2$, $\Col_1$, $\Col_0$, with $d+1$ rows.
In the $i$\ts{th} row of $\Col_m$ we imagine placing $a_{i,m-i}$ stones.
\begin{figure}[ht]
  \centering
  \[
    \begin{array}{cc|cccccccc}
      && \dots & \Col_{d} & \Col_{d-1} & \dots & \Col_3 & \Col_2 & \Col_1 & \Col_0 \\ \hline
      \text{Row $0$} & k_0 & \dots & a_{0,d} & a_{0,d-1} & \dots & a_{0,3} & a_{0,2} & a_{0,1} & a_{0,0} \\
      \text{Row $1$} & q k_1 & \dots & a_{1,d-1} & a_{1,d-2} & \dots & a_{1,2} & a_{1,1} & a_{1,0} \\
      \text{Row $2$} & q^2 k_2 & \dots & a_{2,d-2} & a_{2,d-3} & \dots & a_{2,1} & a_{2,0} \\
      \text{Row $3$} & q^3 k_3 & \dots & a_{3,d-3} & a_{3,d-4} & \dots & a_{3,0} \\
      \multicolumn{1}{c}{\vdots} & \vdots & \dots & \vdots & \vdots & \iddots \\
      \text{Row $d-1$} & q^{d-1} k_{d-1} & \dots & a_{d-1,1} & a_{d-1,0} \\
      \text{Row $d$} & q^d k_d & \dots & a_{d,0} \\ \hline
      \text{Sum} & k & \dots & C_d & C_{d-1} & \dots & C_3 & C_2 & C_1 & C_0
    \end{array}
  \]
  \caption{Admissible decompositions, drawn as an addition table in base-$q$.
    We imagine placing $a_{i,m-i}$ stones in the $i$\ts{th} row of $\Col_m$.}
  \label{fig:addition_table}
\end{figure}

A decomposition corresponds to a way of filling in the digits in
\Cref{fig:addition_table} such that both of the following
capacity constraints hold:
\begin{description}
  \item[Column capacity] We require that summing the addition table in \Cref{fig:addition_table}
    indeed gives
    \[ \sum q^i k_i = \sum C_n q^n. \]

    One would expect \emph{a priori} that there could be some base-$q$ carrying.
    However, \Cref{lem:H1-carry-elimination} will show this case never produces
    optimal values.
    Thus we can regard this condition as saying $\Col_m$ has exactly $C_m$ stones.

  \item[Diagonal capacity]
    For each $n$, let $\Diag_n$ denote the cells labeled $a_{i,n}$ for $0 \le i \le d$.
    Then for every $n$, admissibility requires that
    \begin{equation} \label{eq:source-capacity-H1}
      \text{number of stones in } \Diag_n = \sum_{i=0}^d a_{i,n} \leq q-1.
    \end{equation}
\end{description}

\subsubsection{The objective function $\Phi$}
Finally we compute the objective function.
Using \eqref{eq:intro-bid} and \eqref{eq:intro-degree-contribution},
we see that we are trying to minimize
\begin{align*}
  -\sum_{i=0}^d k_i b(i,d)
  &=\sum_{i=0}^d k_i\left(\frac{q^{d+1}-q^{i+1}}{q-1}+iq^i\right)
  =\frac{q^{d+1}}{q-1}\sum_{i=0}^d k_i
    -\frac{q}{q-1}\sum_{i=0}^d k_iq^i+
     \sum_{i=0}^d iq^ik_i\\
  &=\frac{1}{q-1}
  \left( q^{d+1}\sum_{i=0}^d k_i+(q-1)\sum_{i=0}^d iq^ik_i \right)
  -\frac{qk}{q-1}.
\end{align*}
The final term is fixed once $k$ is fixed.
Hence, if we define
\begin{equation}\label{eq:Phi-H1}
  \begin{aligned}
    g_i &\coloneq q^{d+1-i} + (q-1) \cdot i \\
    \Phi(k_0,\ldots,k_d) &\coloneq q^{d+1}\sum_{i=0}^d k_i+(q-1)\sum_{i=0}^d iq^ik_i
    = \sum_{i,n} q^{n+i} g_i \cdot a_{i,n}
  \end{aligned}
\end{equation}
then \Cref{H1} is equivalent to the following proposition:
\begin{proposition}\label{prop:H1-optimization}
  Let $q=p$, $d\geq 0$, and $k\geq 0$ be fixed.
  Among all admissible decompositions, there is a unique one minimizing $\Phi$
  as defined in \eqref{eq:Phi-H1}.
\end{proposition}
In words, the objective function $\Phi$
is obtained by assigning a weight $q^{n+i} g_i$ to each stone
in the cell labeled $a_{i,n}$ and summing.
We remark that, since $g_i-g_{i+1}=(q-1)(q^{d-i}-1)>0$, the $g_i$ are decreasing:
\begin{equation}\label{eq:g-decreasing-H1}
  g_0>g_1>\cdots>g_d.
\end{equation}

The proof has two parts.
First we show that a minimizer has no carries in the addition table in \Cref{fig:addition_table}.
Then the problem becomes a transparent assignment problem for the individual base-$q$ digit units of $k$.

\subsection{No shifted carries in \Cref{fig:addition_table}}
The next lemma is the first decisive simplification.
\begin{lemma}\label{lem:H1-carry-elimination}
  If $(k_0,\ldots,k_d)$ is an admissible decomposition minimizing $\Phi$, then
  \[ \sum a_{i,m-i} \leq q-1\qquad\text{for every }m.  \]
  Equivalently, $\Col_m$ of \Cref{fig:addition_table} has exactly $C_m$ stones.
\end{lemma}
\begin{proof}
Assume for contradiction that a minimizing admissible decomposition has a carry,
first occurring in $\Col_m$.
We perform the following change:
\begin{itemize}
  \item Choose any $q$ stones in $\Col_m$ and remove them.
  \item Among the rows occupied by these chosen stones,
    let $j$ be maximal subject to $j<d$
    (since $a_{d,m-d} < q$, not all stones are in the final row).
    Add one stone to $a_{j+1,m-j}$.
\end{itemize}
The addition table in \Cref{fig:addition_table} remains valid,
since we removed a total of $q$ stones from $\Col_m$ and added one to $\Col_{m+1}$.
Moreover, the newly added stone was in $\Diag_{m-j}$,
which had at least one stone removed;
hence the diagonal capacity constraints are still satisfied too.

It remains to compare the values of $\Phi$.
The removed stones contribute
\[
  q^m\sum_{\text{chosen}}g_i,
\]
while the inserted stone contributes
\[
  q^{m+1}g_{j+1}=q^m qg_{j+1}.
\]
Hence the replacement lowers $\Phi$ if
\begin{equation}\label{eq:H1-Phi-ineq}
  \sum_{\text{chosen}}g_i > qg_{j+1}.
\end{equation}
Recall from \eqref{eq:g-decreasing-H1} that $g_i$ are decreasing.
By the definition of $j$, all stones appear in either row $d$ or rows $\le j$.
Row $d$ contains $a_{d,m-d} \le q-1$ stones, so
\[ \sum_{\text{chosen}}g_i \ge  g_j+(q-1)g_d.  \]
A direct calculation gives
\begin{align*}
  g_j+(q-1)g_d-qg_{j+1}
  &=\bigl(q^{d+1-j}+(q-1)j\bigr)
    +(q-1)\bigl(q+(q-1)d\bigr)\\
  &\qquad -q\bigl(q^{d-j}+(q-1)(j+1)\bigr)\\
  &=(q-1)^2(d-j)>0.
\end{align*}
Thus the replacement strictly lowers $\Phi$, contradicting minimality.
Therefore no shifted carry can occur in a minimizer.
\end{proof}

\subsection{The assignment problem and its greedy solution}
Thanks to \Cref{lem:H1-carry-elimination},
we can now regard \Cref{fig:addition_table} as a totally combinatorial problem
where $\Col_m$ has exactly $C_m$ stones
and $\Diag_n$ has at most $q-1$ stones,
and we wish to minimize a sum of weights across all the stones.

We can now describe the greedy algorithm concretely and simply:
\begin{algorithm}
  \label{alg:H1-greedy}
  Allocate the stones of $\Col_0$, $\Col_1$, \dots, in that order.
  For the $m$\ts{th} column ($m \ge 0$), we place each of the $C_m$ stones in $\Col_m$
  (one at a time) in the lowest cell that stone can be placed
  in without violating diagonal capacity constraints.
\end{algorithm}

\begin{example}
  Let $q = 11$, $d = 3$, and $k = 8675309_{11}$ (in base eleven).
  Then the diagonal constraint is that each diagonal has at most $10$ stones in it.
  The corresponding greedy output is:
  \[
    \begin{array}{c|ccc cccc}
      & \Col_6 & \Col_5 & \Col_4 & \Col_3 & \Col_2 & \Col_1 & \Col_0 \\ \hline
      \text{Row $0$} & 0 & 0 & 0 & 0 & 0 & 0 & 9 \\
      \text{Row $1$} & 0 & 0 & 0 & 0 & 2 & 0 \\
      \text{Row $2$} & 0 & 0 & 4 & 5 & 1 \\
      \text{Row $3$} & 8 & 6 & 3 & 0 \\ \hline
      k & 8 & 6 & 7 & 5 & 3 & 0 & 9
    \end{array}
  \]
\end{example}

\begin{lemma}
  The greedy assignment described in \Cref{alg:H1-greedy} minimizes $\Phi$.
\end{lemma}
\begin{proof}
  Let $a_{i,n}$ be the allocation minimizing $\Phi$.
  assume for contradiction it doesn't match the greedy allocation,
  and let $\Col_m$ be the rightmost column where they don't match.
  Choose the largest $j$ such that $a_{j,m-j}$ has fewer stones than the greedy allocation;
  then there exists some $i < j$ such that $a_{i,m-i} > 0$.
  We consider two cases:
  \begin{itemize}
  \item Suppose $\Diag_{m-j}$ is not at full capacity.
  Then we are free to move a stone directly downwards from $a_{i,m-i}$ to $a_{j,m-j}$.
  This changes the value of $\Phi$ by
  \[ q^{m} (g_j - g_i) < 0 \]
  due to \eqref{eq:g-decreasing-H1}, contradicting minimality.

  \item Suppose $\Diag_{m-j}$ is at full capacity already.
  This implies the existence of a $t > 0$ such that $a_{j+t,m-j} > 0$.
  Thus, we arrive at a picture like in \Cref{fig:parallelogram}.

  \begin{figure}[b]
    \begin{tikzpicture}[scale=2]
      \node (A) at (1,2)  {$\boxed{a_{i,m-i} > 0}$};
      \node (B) at (0,1)  {$a_{i+t, m-i}$};
      \node (C) at (0,0)  {$\boxed{a_{j+t, m-j} > 0}$};
      \node (D) at (1,1)  {$a_{j,m-j}$};
      \draw (A) -- (B) -- (C) -- (D) -- (A);
    \end{tikzpicture}
    \caption{An impossible parallelogram that we show can't appear in an optimal solution.}
    \label{fig:parallelogram}
  \end{figure}
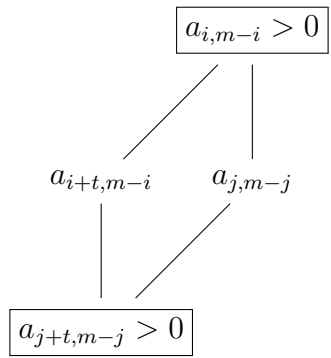

  However, we prove that \Cref{fig:parallelogram} violates minimality by moving two stones:
  we remove one stone from each of $a_{i,m-i}$ and $a_{j+t,m-j}$
  and add one stone to $a_{i+t,m-i}$ and $a_{j,m-j}$ instead.
  This has no effect on column or diagonal sums, but it changes the value of $\Phi$ by
  \begin{align*}
    q^{m} (g_j - g_i) - q^{m+t} (g_{j+t}-g_{i+t})
    &= q^m \left[ (q^{d+1-j} - q^{d+1-i}) + (q-1)(j-i) \right] \\
    &\qquad - q^{m+t} \left[ (q^{d+1-(j+t)} - q^{d+1-(i+t)}) + (q-1)(j-i) \right] \\
    &= (q-1)(j-i)(q^m-q^{m+t}) < 0. \qedhere
  \end{align*}
  \end{itemize}
\end{proof}

The proof of \Cref{H1} is now just a translation from the minimization of $\Phi$ back
to the degree of the corresponding term in Carlitz's expansion.
\begin{proof}[Proof of \Cref{thm:intro-H1}]
The nonzero summands in Carlitz's expansion are exactly the admissible decompositions.
By \Cref{prop:H1-optimization},
the degree contribution has a unique maximum among those decompositions.
This is precisely \Cref{H1}.
\end{proof}

\section{The common expansion of $S_d(k)$ in the proofs of H2 and H3}
\label{sec:H2H3}

The proof (and indeed even the statement) of H1 was based on
Carlitz's expansion \eqref{eq:carlitz-full}.
In contrast, the proofs of H2 and H3 will not use \eqref{eq:carlitz-full}
but instead a more direct expansion that we describe now in \Cref{lem:multinomial_expansion}.
Because this section is used in H3, we will work in the generality of $q = p^f$
without assuming that $f = 1$.

\subsection{The expansion}
To state \Cref{lem:multinomial_expansion}, we need the following definition:
\begin{definition}
  For $d \ge 1$ and $k > 0$, define $\Tdk$ to be the set of $d$-tuples
  $\mathbf{m} = (m_1,\ldots,m_d)$ such that
  \[
    m_i>0\qquad {\text {\rm and}}\qquad (q-1)\mid m_i
    \quad(1\leq i\leq d),
  \]
  and
  \[
    (k-1) \oplus m_1\oplus\cdots\oplus m_d
  \]
  is carry-free in base-$p$.

  Note the set $\Tdk$ is always nonempty.
  (Indeed, choose $d$ disjoint sufficiently high blocks of $f$ consecutive base-$p$ digits,
  far beyond the support of $k-1$, and place a copy of $q-1=p^f-1$ shifted into each block.)
\end{definition}

We show that the set $\Tdk$ records the possible exponent patterns that survive
both Lucas' theorem and the finite-field summation.
\begin{lemma}
  [Multinomial expansion of $S_d(k)$]
  \label{lem:multinomial_expansion}
  Fix $k > 0$ and $d \ge 1$.
  Then
  \[ S_d(k) = t^{-dk} \sum_{\mathbf{m} \in \Tdk}
    c_{\mathbf{m}} t^{-(m_1+2m_2+\dots+dm_d)} \]
  for some nonzero scalars $c_{\mathbf{m}} \in \F_q^\times$.
\end{lemma}

\begin{proof}
  Write a monic polynomial of degree $d$ as
  \[
    a=t^d+\theta_1t^{d-1}+\cdots+\theta_d\qquad \text{and} \qquad \theta_i\in\F_q.
  \]
  Then
  \[
    a^{-k}=t^{-dk}(1+\theta_1t^{-1}+\cdots+\theta_dt^{-d})^{-k}.
  \]
  Expanding the last factor gives
  \begin{align*}
  S_d(k)
  &=t^{-dk}
    \sum_{\theta_1,\ldots,\theta_d\in\F_q}
    \sum_{m_1,\ldots,m_d\geq 0}
    \binom{-k}{m_1+\cdots+m_d}
    \binom{m_1+\cdots+m_d}{m_1,\ldots,m_d}\\
  &\qquad\qquad\qquad\cdot
    \theta_1^{m_1}\cdots\theta_d^{m_d}
    t^{-(m_1+2m_2+\cdots+dm_d)}.
  \end{align*}
  By \Cref{lem:field-sums}, summing over $\theta_i\in\F_q$ kills every term except those with
  \[
    m_i>0\qquad {\text {\rm and}}\qquad (q-1)\mid m_i.
  \]
  Next we turn attention to the binomial and multinomial coefficient.
  Put $y \coloneq m_1 + \dots + m_d$.
  The identity
  \[
    \binom{-k}{y}=(-1)^y\binom{k+y-1}{y}
  \]
  shows, by \Cref{lem:Lucas-binomial},
  that $\binom{-k}{y}$ is nonzero modulo $p$ if and only if $(k-1)+y$ is carry-free.
  By \Cref{lem:mult-Lucas},
  the multinomial coefficient is nonzero modulo $p$ if and only if $m_1+\cdots+m_d=y$ is carry-free.
  These two carry-free conditions hold simultaneously if and only if the combined addition
  \[
    (k-1)\oplus m_1\oplus\cdots\oplus m_d
  \]
  is carry-free, i.e., if $(m_1, \dots, m_d) \in \Tdk$.
\end{proof}

From \Cref{lem:multinomial_expansion} it is already clear that
\begin{equation}
  s_d(k) \ge dk + \min_{\mathbf{m} \in \Tdk} (m_1 + 2m_2 + \dots + dm_d).
  \label{eq:sdk_min}
\end{equation}

However, we are going to use Sheats' theorem to prove that in fact
the minimizing tuple $\mathbf{m}$ is unique;
this means equality must hold in \eqref{eq:sdk_min}.

\subsection{Sheats' uniqueness theorem}
For $m\geq 1$ and $N>0$, let $U_m(N)$ be the set of tuples
\[
  X=(X_1,\ldots,X_m)\in\Z_{\geq 0}^m
\]
such that
\[
  N=X_1\oplus X_2\oplus\cdots\oplus X_m
\]
is carry-free in base-$p$, and
\[
  X_j>0\qquad {\text {\rm and}}\qquad (q-1)\mid X_j
  \qquad(1\leq j\leq m-1).
\]
There is no positivity or divisibility condition on $X_m$.  Define
\[
  \wt(X) \coloneq X_1+2X_2+\cdots+mX_m.
\]
We use Sheats' theorem in the following form.
It provides the uniqueness input needed to prevent cancellation at the minimal
exponent in the positive-power expansion.

\begin{theorem}[{Sheats, \cite[Lemma~1.3]{Sheats1998}}]\label{thm:H3-Sheats}
If $U_m(N)$ is nonempty, then $U_m(N)$ contains a unique element maximizing $\wt(X)$.
\end{theorem}

\begin{remark}\label{rem:H3-Sheats-translation}
\cite[Lemma~1.3]{Sheats1998} also asserts that if $U_m(N)$ is nonempty,
the maximizing element is the so-called \emph{greedy element}.
In Sheats' terminology, this greedy element is obtained by making the successive
locally extremal choices among the remaining carry-free summands subject to the divisibility conditions.
The proof below uses only uniqueness of the maximizer, not the description of the maximizer as greedy.
\end{remark}

For our purposes it is more convenient to reverse the order of the components in Sheats' theorem.
The following proposition is only this reindexing, but it puts the statement in the form used later.

\begin{corollary}[Reindexed Sheats uniqueness]\label{cor:H3-sheats-reindexed}
Let $Y\geq 0$ and $d\geq 1$.  Consider decompositions
\begin{equation}\label{eq:H3-decomp-Y}
  Y=m_0\oplus m_1\oplus\cdots\oplus m_d
\end{equation}
with
\[
  m_0\geq 0,
  \qquad
  m_i>0,
  \qquad
  (q-1)\mid m_i
  \quad(1\leq i\leq d).
\]
If such decompositions exist, then among them the weighted sum
\[
  W(m_1,\ldots,m_d) \coloneq m_1+2m_2+\cdots+dm_d
\]
has a unique minimum.
\end{corollary}

\begin{proof}
Associate to \eqref{eq:H3-decomp-Y} the $(d+1)$-tuple
\[
  X \coloneq (m_d,m_{d-1},\ldots,m_1,m_0).
\]
This gives a bijection between the decompositions \eqref{eq:H3-decomp-Y} and the set $U_{d+1}(Y)$.
Under this bijection,
\begin{align*}
  \wt(X)
  &=m_d+2m_{d-1}+\cdots+d m_1+(d+1)m_0.
\end{align*}
Since $Y=m_0+m_1+\cdots+m_d$, we have
\begin{align*}
  (d+1)Y-\wt(X)
  &=m_1+2m_2+\cdots+dm_d\\
  &=W(m_1,\ldots,m_d).
\end{align*}
Thus minimizing $W$ is the same as maximizing $\wt(X)$.
Uniqueness follows from \Cref{thm:H3-Sheats}.
\end{proof}

\subsection{The formula for $s_d(k)$ using $\Tdk$}
We can now prove equality always holds in \eqref{eq:sdk_min}.
This formula is a streamlined version of the positive-power minimization
implicit in Thakur's use of Sheats' theorem in \cite{Thakur2009};
the proof is included to make the dependence on Sheats' uniqueness explicit.
Sheats' uniqueness is used only at the end,
to ensure that the minimal exponent contributes a nonzero leading coefficient.

\begin{lemma}[Positive-power minimization formula]\label{lem:H3-minformula}
  For every $q=p^f$, $d\geq 1$, and $k>0$, define
  \[ M_d(k-1) \coloneq \min_{\mathbf{m} \in \Tdk} (m_1 + 2m_2 + \dots + dm_d). \]
  Then we have
  \begin{equation}\label{eq:H3-minformula}
    s_d(k) = dk + M_d(k-1).
  \end{equation}
\end{lemma}

\begin{proof}
We saw already that $\Tdk$ is nonempty,
and that $s_d(k) \ge dk + M_d(k-1)$ by \Cref{lem:multinomial_expansion}.
To show equality holds, it suffices to rule out cancellation among terms of minimal weight.
We will use \Cref{cor:H3-sheats-reindexed} to show in fact
there is only one term of minimal weight, so cancellation can never occur.

Choose $N$ so large that the base-$p$ digits of $k-1$ and of all tuples vanish in positions $\geq N$.
All complements in the next paragraph are taken only in the digit positions $0,1,\ldots,N-1$.
Fix
\[ Y \coloneq p^N-1-(k-1).  \]
For each $d$-tuple $(m_1, \dots, m_d) \in \Tdk$, we define
\[ m_0 \coloneq Y-(m_1+\cdots+m_d) \ge 0.  \]
We contend this produces a bijection
\begin{align*}
  \Tdk &\longleftrightarrow \text{decompositions in \Cref{cor:H3-sheats-reindexed}} \\
  (m_1, \dots, m_d) &\longmapsto (m_0, m_1, \dots, m_d)
\end{align*}
Indeed, the carry-free condition
$(k-1)\oplus m_1\oplus\cdots\oplus m_d$
means that, in every digit position below $N$,
the digit sum of $m_1+\cdots+m_d$ is at most the corresponding digit of the complement $p^N-1-(k-1)$,
so that $Y=m_0\oplus m_1\oplus\cdots\oplus m_d$ is carry-free as well.

Because \Cref{cor:H3-sheats-reindexed} promises a unique minimizer,
this now follows for $\Tdk$ too.
\end{proof}

\section{Proof of \Cref{H2} over prime fields}\label{sec:H2}
We again assume in this section that $q=p$ is prime.
We will prove \Cref{H2} using \Cref{lem:H3-minformula}.

\subsection{The reciprocal slot formula for \texorpdfstring{$s_d(k)$}{s d(k)}}
For $k>0$, write
\[ k-1=\sum_{e\geq 0}a_ep^e, \qquad 0\leq a_e\leq p-1.  \]
Define the \emph{complementary digit-slot multiset}
\begin{equation}\label{eq:H2-Ak}
  \Slots_p(k-1) \coloneq \{p^e\text{ repeated }p-1-a_e\text{ times}\}_{e\geq 0}.
\end{equation}
This multiset is infinite, because $a_e=0$ for all sufficiently large $e$.
The interpretation is simple:
$\Slots_p(k-1)$ lists the base-$p$ digit slots that can be added to $k-1$ without producing a carry.

Recalling our block notation from \Cref{sec:block}, we choose
\begin{align*}
  n &\coloneq p-1 \\
  M &= \Slots_p(k-1)
\end{align*}
and abbreviate
\[
  B_r(k-1) \coloneq B_r(\Slots_p(k-1))\qquad {\text {\rm and}}\qquad \beta_r(k-1) \coloneq \beta_r(\Slots_p(k-1)).
\]
We now give the main result of this section,
a formula for $s_d(k)$ in terms of these blocks.
\begin{theorem}[Reciprocal slot formula]\label{thm:H2-slot-formula}
For $q=p$, $d\geq 1$, and $k>0$,
\begin{equation}\label{eq:H2-slot-formula}
  s_d(k)=dk+d\beta_1(k-1)+(d-1)\beta_2(k-1)+\cdots+\beta_d(k-1).
\end{equation}
\end{theorem}

\begin{proof}
  In the notation of \Cref{lem:H3-minformula},
  it suffices to compute $M_d(k-1)$.
  We identify each $m_i$ with a sum of powers of $p$ given by its base-$p$ representation;
  let \[ T_i \subseteq \Slots_p(k-1) \] be the corresponding powers of $p$
  (so that $m_i = \sum_{c \in T_i} c$).
  Then the requirement $\mathbf{m} \in \Tdk$ translates as:
  \begin{itemize}
    \item The condition $n \mid m_i$ and $m_i > 0$
      means each $T_i$ is nonempty and $|T_i|$ is divisible by $n$
      (because $c \equiv 1 \pmod{p-1}$ for all $c \in T_i$).
    \item The carry-free condition says that the $T_i$ are pairwise
      disjoint submultisets of $\Slots_p(k-1)$.
  \end{itemize}
  If $|T_i|\geq 2n$ for any $i$, removing the largest $n$ slots from $T_i$
  leaves a nonempty set of cardinality divisible by $n$ and strictly lowers the exponent.
  Thus, we may as well assume in fact that $|T_i| = n$ for every $i$.

  The block-rearrangement inequality \Cref{lem:H2-block-min} now gives
  \[ M_d(k-1) = \min_{\substack{|T_i|=n \\ \bigsqcup T_i \subseteq \Slots_p(k-1)}}
    \left( \sum_{c \in T_1} c + 2 \sum_{c \in T_2} c + \dots
    \right) = d\beta_1(k-1)+\dots+\beta_d(k-1) \]
  which is what we wanted to prove.
\end{proof}

\begin{remark}
  We can also prove the uniqueness of the minimizer directly in this case
  without having to use Sheats' theorem, as follows.
  The equality statement in \Cref{lem:H2-block-min} forces the numerical tuple
  \[
    (m_1,\ldots,m_d)=(\beta_d(k-1),\beta_{d-1}(k-1),\ldots,\beta_1(k-1)).
  \]
  Interchanging equal-valued slots may change the named submultisets $T_i$,
  but it does not change this numerical tuple.
  Hence exactly one numerical tuple contributes at the minimal exponent.

  This means that \Cref{thm:H2-slot-formula} could be proved directly
  from \Cref{lem:multinomial_expansion}, without depending on Sheats' theorem.
\end{remark}

\begin{example}
  We show \Cref{thm:H2-slot-formula} for $p=q=2$, $k=5$ and $d=1$.
  Then $k-1 = 4 = 100_2$, and $\Slots_2(4)=\{1,2,8,16,32,\ldots\}$.
  Hence $\beta_1(4) = 1$, $\beta_2(4) = 2$, $\beta_3(4) = 8$, etc., and
  \[ s_1(5) = 5 \cdot 1 + \beta_1(4) = 5 \cdot 1 + 1 = 6.  \]
  As further examples, we show $p=q=2$, $k \in \{6,7,8\}$, and $d = 2$ in \Cref{tab:H2_example}.
\end{example}

\begin{table}[ht]
  \[
    \begin{array}{c ccc}
      k & k-1 & \Slots_2(k-1) & s_2(k) \\ \hline
      k = 6 & 5 = 101_2 & \Slots_2(5) = \{2,8,16,32,\dots\}
        & s_2(6) = 2 \cdot 6 + 2 \beta_1(5) + \beta_2(5) = 24  \\
      k = 7 & 6 = 110_2 & \Slots_2(6) = \{1,8,16,32,\dots\}
        & s_2(7) = 2 \cdot 7 + 2 \beta_1(6) + \beta_2(6) = 24  \\
      k = 8 & 7 = 111_2 & \Slots_2(7) = \{8,16,32,\dots\}
        & s_2(8) = 2 \cdot 8 + 2 \beta_1(7) + \beta_2(7) = 48
    \end{array}
  \]
  \caption{Examples of the slot formula for $p=q=2$, $d=2$, and $6 \le k \le 8$.}
  \label{tab:H2_example}
\end{table}

The case $d=1$ is worth isolating because it identifies the first block with the shift from $k$ to $s_1(k)$,
which is the shift appearing in H2.

\begin{corollary}[The first reciprocal block]\label{cor:H2-s1}
Assume $q=p$ and let $k>0$.  In degree one,
the leading exponent is obtained by using the first admissible reciprocal block.  Equivalently,
\[
  s_1(k)=k+\beta_1(k-1).
\]
\end{corollary}

\subsection{The H2 minimization}
We now use the slot formula we just proved (\Cref{thm:H2-slot-formula}) to deduce \Cref{H2}.
The admissible integer $j$ in H2 corresponds to extracting a submultiset from the complementary slots of $k$;
the extraction inequality says that the first block is the unique optimal extraction.

\begin{proof}[Proof of \Cref{thm:intro-H2}]
If $d=1$, then $s_0(\ell)=0$ for every $\ell>0$, so the quantity in H2 is $s_1(k)+j$.
Obviously this is minimized uniquely at $j=0$.

Assume now that $d\geq 2$. Abbreviate
\[
  M \coloneq \Slots_p(k-1) \qquad{\text{and}}\qquad \beta_r \coloneq \beta_r(M).
\]
By \Cref{cor:H2-s1}, we have
\[
  s_1(k)=k+\beta_1.
\]
For an admissible $j$, we have the requirement
\[ 0 \not\equiv \binom{s_1(k)+j-1}{k-1} = \binom{(k-1)+u}{k-1} \pmod p
  \quad\text{where} \quad u \coloneq \beta_1+j.  \]
By \Cref{lem:Lucas-binomial},
we may identify $u$ via its base-$p$ digits
with a unique finite submultiset $U \subseteq M$.
Since $n \mid \beta_1$ and $n \mid j$, we still have $n \mid u$ and thus as before $|U|$
is a positive multiple of $n$.

The slots used by $u$ are precisely removed from the complement
when one passes from $k$ to $s_1(k)+j=k+u$, meaning
\[
  \Slots_p((k-1)+u) = M \setminus U.
\]
Using \Cref{thm:H2-slot-formula} with $d-1$ in place of $d$, we obtain
\begin{align*}
  s_{d-1}(s_1(k)+j)+s_1(k)+j
  &= s_{d-1}(k+u) + (k+u) \\
  &= \left( (d-1)(k+u) + \Phi_{d-1}(M \setminus U) \right) + (k+u) \\
  &= d(k+u)+\Phi_{d-1}(M\setminus U)\\
  &= dk+\left(d\sum_{c\in U}c+\Phi_{d-1}(M\setminus U)\right)
  \ge dk + \Phi_d(M)
\end{align*}
with the last inequality by the extraction inequality \Cref{lem:H2-extraction}.

We need to verify the only equality case is $j = 0$.
For $j=0$, we have $u=\beta_1$, and $U$ is exactly the first block $B_1(M)$; hence equality holds.
If $j>0$, then $u>\beta_1$, so $U$ cannot be the first block $B_1(M)$ as a multiset of values.
The equality condition in \Cref{lem:H2-extraction} is impossible, and the inequality is strict.
Thus the minimum is attained uniquely at $j=0$.
\end{proof}

Finally we translate the H2 minimization statement into the recursion.
The point is that the unique minimum in the reindexed expansion is represented by the $j=0$ term.

\begin{proof}[Proof of \Cref{cor:intro-H2}]
Insert \Cref{thm:intro-H2} into expansion \eqref{eq:intro-H2-expansion},
reindexed as in \cite[Section 3.2]{Thakur2009}.
The unique minimal term is the term $j=0$, so its leading coefficient cannot cancel.  This gives
\[
  s_d(k)=s_{d-1}(s_1(k))+s_1(k). \qedhere
\]
\end{proof}

\begin{example}\label{ex:H2-q2}
Let $q=2$, $k=5$, and $d=3$, so $s_1(5) = 6$.
Thus \Cref{H2} asserts that among the admissible $j\geq 0$, the expression
\[ s_2(6+j)+6+j \]
is uniquely minimized at $j=0$.
For comparison, the values $0 \le j \le 2$
are admissible since $\binom{j+5}{k-1} = \binom{j+5}{4}$
are all odd for those $j$; and from \Cref{tab:H2_example} we get
\[ s_2(6)+6 = 30, \qquad s_2(7)+7 = 31, \qquad s_2(8)+8 = 56. \]
\end{example}

\subsection{Deducing the consequences of H2}
Our introduction advertised several nice corollaries of \Cref{H2};
we make good on that advertisement here.

We first derive the Newton-polygon convexity consequence from the recursion alone.

\begin{proof}[Proof of \Cref{cor:intro-convexity}]
The recursion reduces the monotonicity statement for $d$ to the same statement
with $d$ replaced by $d-1$ and $k$ replaced by $s_1(k)$.  The base case is $s_1(k)>s_0(k)=0$.
The same reduction applies to the jump inequality.  It remains only to check the base case
\[
  2s_1(k)<s_0(k)+s_2(k)=s_2(k).
\]
By \eqref{eq:intro-recursion}, this is equivalent to $s_1(k)<s_1(s_1(k))$,
which follows from the fact that $s_1(m)>m$ for every $m>0$.
\end{proof}

The following consequence records the Riemann Hypothesis implication supplied by the strict jump
inequalities through Thakur's Newton-polygon criterion.

\begin{proof}[Proof of \Cref{cor:intro-RH}]
By \Cref{cor:intro-convexity}, the sequence $s_d(k)$ has strictly increasing successive differences.
This is precisely the Newton-polygon hypothesis used in Thakur's proof of the Riemann-hypothesis analogue;
see \cite[Theorem~3]{Thakur2009} and the surrounding discussion in \cite[Section~5.8]{Thakur2004}.
The claimed zero-distribution statement therefore follows from that criterion.
This recovers the prime-field case,
already known from the work of Wan and Diaz--Vargas and included in Sheats' general theorem.
\end{proof}

For multizeta values, the same monotonicity singles out the lowest admissible
degree pattern as the unique leading contribution.

\begin{proof}[Proof of \Cref{cor:intro-MZV}]
By \Cref{cor:intro-convexity},
the quantities $s_d(k)$ are strictly increasing in $d$ for every positive $k$.
Therefore, in the defining sum for $\zeta(s_1,\ldots,s_r)$, the term of largest $t$-degree is unique:
it is the term with the smallest allowable degree indices.
Since this leading term has no competitor of the same degree, it cannot cancel.
Hence the multizeta value is nonzero, as in \cite[Theorem 4]{Thakur2009}.
\end{proof}

\section{Proof of \Cref{H3} over arbitrary finite fields}
\label{sec:H3}
In this section $q = p^f$ for some $f \ge 1$.
We can directly prove \Cref{H3} from \Cref{lem:H3-minformula}.

\begin{proof}[Proof of \Cref{thm:intro-H3}]
By \Cref{lem:H3-minformula},
\[
  s_d(k)=dk+M_d(k-1),
  \qquad
  s_d(k+1)=d(k+1)+M_d(k).
\]
Therefore
\begin{equation}\label{eq:H3-difference}
  s_d(k+1)-s_d(k)=d+M_d(k)-M_d(k-1).
\end{equation}
If $p\nmid k$, then the units digit of $k$ in base-$p$ is nonzero.
Passing from $k$ to $k-1$ lowers only that units digit by $1$ and leaves all higher digits unchanged.
Hence
\[
  k\oplus m_1\oplus\cdots\oplus m_d \text{ carry-free}
  \implies (k-1)\oplus m_1\oplus\cdots\oplus m_d \text{ carry-free}.
\]
This implies $\mathcal T_{d,k} \subseteq \Tdk$ and we have
\[ M_d(k)\geq M_d(k-1).  \]
Hence \eqref{eq:H3-difference} gives $s_d(k+1)-s_d(k)\geq d>0$ and \Cref{H3} is proved.
\end{proof}

\section{Appendix: Lean formalizations generated by AxiomProver}
\label{sec:axiomprover}
Here we provide the context for this project as well as the protocol used for Lean
formalization and verification (see \cite{Mathlib2020, Lean}).
The formal proofs provided in this work were developed and verified using Lean~4.28.0.
Compatibility with earlier or later versions is not guaranteed due to the
evolving nature of the Lean 4 compiler and its core libraries.

\subsection{Description of artifacts}
The relevant files are all posted in the following repository:
\begin{center}
  \url{https://github.com/AxiomMath/zeta-h123}
\end{center}

\begin{figure}[ht]
  \centering
  \begin{tikzpicture}[x=1cm, y=1cm]
    \node (lucas) [proved] at (0,6) {Lucas (\S\ref{sec:lucas}) \\ Proved within \texttt{Lem41}};
    \node (pow) [proved] at (5,6) {Pow sum (\S\ref{sec:power_sum}) \\ Proved within \texttt{Lem41}};
    \node (block) [proved] at (10,6) {Block ineq (\S\ref{sec:block}) \\ Proved within \texttt{H2}};
    \node (lem41) [proved, text width = 5cm] at (2.5,4) {Expansion (\Cref{lem:multinomial_expansion}) \\ \texttt{Lem41}};
    \node (minform) [assumed, text width = 5cm] at (2.5,2) {Min formula (\Cref{lem:H3-minformula})};
    \node (H1) at (0,0) [proved] {\Cref{H1} (\S\ref{sec:H1}) \\ \texttt{H1}};
    \node (H2) at (5,0) [proved] {\Cref{H2} (\S\ref{sec:H2}) \\ \texttt{H2}};
    \node (H3) at (10,0) [proved] {\Cref{H3} (\S\ref{sec:H3}) \\ \texttt{H3}};
    \draw [arrow] (lucas) -- (lem41);
    \draw [arrow] (pow) -- (lem41);
    \draw [arrow] (lem41) -- node[anchor=west] {using Sheats' theorem} (minform);
    \draw [arrow] (minform) -- (H2);
    \draw [arrow] (minform) -- (H3);
    \draw [arrow] (block) -- (H2);
  \end{tikzpicture}
  \caption{A copy of \Cref{fig:skeleton}
    additionally annotated with the locations of the formal proofs in our repository.
    We do not formalize \Cref{lem:H3-minformula}
    because of the external dependence on Sheats' theorem (\Cref{thm:H3-Sheats}),
    but prove all the remaining results assuming it.}
  \label{fig:skeleton_with_formalization}
\end{figure}
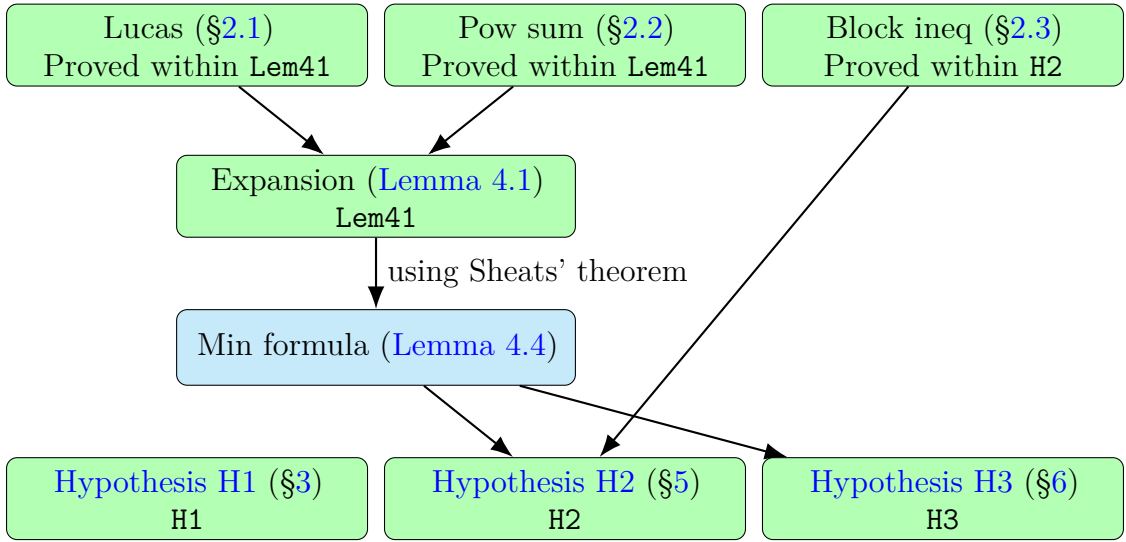

The formalizations provided in this repository consist of:
\begin{itemize}
  \item \texttt{H1/}: A proof of \Cref{H1}.
    This can thus be thought of as a translation of \Cref{sec:H1} into Lean.

  \item \texttt{H2/}: A proof of \Cref{H2} where $s_d(k)$ is (re-)defined
    according to \Cref{lem:H3-minformula},
    thus rephrasing the statement to be combinatorial only (not requiring $S_d(k)$).
    This can thus be thought of as a translation of \Cref{sec:H2}
    (as well as \Cref{sec:block}) into Lean.

  \item \texttt{H3/}: A proof of \Cref{H3} with the same use of \Cref{lem:H3-minformula}.
    This can thus be thought of as a translation of the
    (quite short) \Cref{sec:H3} into Lean.

  \item \texttt{Lem41/}: A formalization of \Cref{lem:multinomial_expansion},
    which provided the link between the original $S_d(k)$ to $\Tdk$.
    This required formalizing \Cref{sec:lucas} and \Cref{sec:power_sum} as well.
\end{itemize}
See \Cref{fig:skeleton_with_formalization} for a diagram of the dependencies
and where these dependencies end up being formalized in the repository.
Note that, because \Cref{lem:H3-minformula} has an
external dependence on Sheats' theorem (\Cref{thm:H3-Sheats}) in \Cref{sec:H2H3},
we did not include this as a task in formalization.

\subsection{Input files}
Next we describe the inputs to AxiomProver used to generate the above formalizations.
For each of the four formalizations, we had the following input files:
\begin{itemize}
  \item \texttt{task.md}: a self-contained statement of the result to be proved
  \item \texttt{informal.tex}: a self-contained natural-language proof
    edited from an earlier draft of this paper.
\end{itemize}
These input files are hosted in the GitHub repository above.
We emphasize that the task for AxiomProver was merely to formalize the proofs in this paper,
rather than to invent the arguments itself.
That is, \texttt{informal.tex} contained the proof to be formalized.
Therefore, we do not consider these proofs to be end-to-end unassisted automated solving.

Based on these input files, AxiomProver generated two output files:
\begin{itemize}
  \item \texttt{problem.lean}, a Lean 4.28.0 formalization of the problem statement; and
  \item \texttt{solution.lean}, a complete Lean 4.28.0 formalization of the proof.
\end{itemize}

This paper itself is written by the authors for human readers;
the use of AI was restricted only to proofreading.
(A diff of the proofreading edits is available upon request.)
At first glance, the proofs found by AxiomProver may not resemble the narrative presented in this paper.
Turning a Lean file into a human-readable proof is difficult
because Lean is written as code for a type-checker.

\section*{Acknowledgements}\noindent The authors thank Dinesh Thakur for suggesting that AxiomProver
be tested against Hypotheses H1, H2, and H3.
We also thank Simon Mahns and Karun Ram for their assistance in the preparation of this paper.

\end{document}